\documentclass[twoside,leqno, british,11pt]
{amsart}
\usepackage{amsmath, amssymb, amsbsy, amsthm}
\usepackage{babel,eucal,url,amssymb,
enumerate,amscd,
}

\usepackage[T1]{fontenc}

\begin{document}

\title[Curvature properties of two Naveira classes]
{Curvature properties of two Naveira classes of Riemannian product
manifolds}

\author[D. Gribacheva]{Dobrinka Gribacheva}


\frenchspacing

\newcommand{\ie}{i.e. }
\newcommand{\X}{\mathfrak{X}}
\newcommand{\W}{\mathcal{W}}
\newcommand{\F}{\mathcal{F}}
\newcommand{\T}{\mathcal{T}}
\newcommand{\LL}{\mathcal{L}}
\newcommand{\TT}{\mathfrak{T}}
\newcommand{\M}{(M,\f,\xi,\eta,g)}
\newcommand{\Lf}{(G,\f,\xi,\eta,g)}
\newcommand{\R}{\mathbb{R}}
\newcommand{\s}{\mathfrak{S}}
\newcommand{\n}{\nabla}
\newcommand{\tr}{{\rm tr}}
\newcommand{\Div}{{\rm div}}
\newcommand{\nn}{\tilde{\nabla}}
\newcommand{\tg}{\tilde{g}}
\newcommand{\f}{\varphi}
\newcommand{\D}{{\rm d}}
\newcommand{\id}{{\rm id}}
\newcommand{\al}{\alpha}
\newcommand{\bt}{\beta}
\newcommand{\gm}{\gamma}
\newcommand{\dt}{\delta}
\newcommand{\lm}{\lambda}
\newcommand{\ta}{\theta}
\newcommand{\om}{\omega}
\newcommand{\Om}{\Omega}
\newcommand{\ea}{\varepsilon_\alpha}
\newcommand{\eb}{\varepsilon_\beta}
\newcommand{\eg}{\varepsilon_\gamma}
\newcommand{\sx}{\mathop{\mathfrak{S}}\limits_{x,y,z}}
\newcommand{\norm}[1]{\left\Vert#1\right\Vert ^2}
\newcommand{\nf}{\norm{\n \f}}
\newcommand{\Span}{\mathrm{span}}
\newcommand{\grad}{\mathrm{grad}}
\newcommand{\thmref}[1]{The\-o\-rem~\ref{#1}}
\newcommand{\propref}[1]{Pro\-po\-si\-ti\-on~\ref{#1}}
\newcommand{\secref}[1]{\S\ref{#1}}
\newcommand{\lemref}[1]{Lem\-ma~\ref{#1}}
\newcommand{\dfnref}[1]{De\-fi\-ni\-ti\-on~\ref{#1}}
\newcommand{\corref}[1]{Corollary~\ref{#1}}



\numberwithin{equation}{section}
\newtheorem{thm}{Theorem}[section]
\newtheorem{lem}[thm]{Lemma}
\newtheorem{prop}[thm]{Proposition}
\newtheorem{cor}[thm]{Corollary}

\theoremstyle{definition}
\newtheorem{defn}{Definition}[section]

\hyphenation{Her-mi-ti-an ma-ni-fold ah-ler-ian}




\begin{abstract}
The main aim of the present work is to obtain some curvature
properties of the manifolds from two classes of Riemannian product
manifolds. These classes are two basic classes from Naveira
classification of Riemannian almost product manifolds.
\end{abstract}



\maketitle

\begin{center}
\end{center}

\section*{Introduction}

A Riemannian almost product manifold $(M, P, g)$ is a
differentiable manifold $M$ with an almost product structure $P$
and a Riemannian metric $g$ such that $P^2x = x$ and $g(Px, Py) =
g(x, y)$ for any tangent vectors $x$ and $y$.

K. Yano initiated in \cite{Ya} the study of Riemannian almost
product manifolds. A.M. Naveira gave in \cite{Na} a classification
of these manifolds with respect to the covariant derivative $\n
P$, where $\n$ is the Levi-Civita connection of $g$. This
classification is very similar to the Gray-Hervella classification
in \cite{GrHe} of almost Hermitian manifolds. In \cite{StaGri} M.
Staikova and K. Gribachev have obtained a classification of the
Riemannian almost product manifolds for which $\tr P = 0$. In this
case the manifold is even-dimensional. The basic class $\W_1$ from
the Staikova-Gribachev classification is the class of conformal
Riemannian $P$-manifolds or shortly $\W_1$-manifolds. It is an
analogue of the class of conformal K\"ahler manifolds in almost
Hermitian geometry. It is valid $\W_1
=\overline{\W}_3\oplus\overline{\W}_6$, where  $\overline{\W}_3$
and $\overline{\W}_6$ are basic classes from the Naveira
classification. In some sense these manifolds have dual
geometries.

The main aim in the present work\footnote{Partially supported by
project NI11-FMI-004 of the Scientific Research Fund, Paisii
Hilendarski University of Plovdiv, Bulgaria} is to obtain some
curvature properties of the manifolds from the Naveira classes
$\overline{\W}_3$ and $\overline{\W}_6$ with respect to $\n$.

The paper is organized as follows. In Sec. 1 we give necessary
facts about Riemannian almost product manifolds, the classes
$\W_1$, $\overline{\W}_3$, $\overline{\W}_6$ and the notion a
Riemannian $P$-tensor on a Riemannian almost product manifold,
which is an analogue of the notion of a K\"ahler tensor in
Hermitian geometry. In Sec. 2 we obtain curvature properties of
the $\overline{\W}_3$-manifolds with respect to $\n$. Using the
same methods as in Sec. 2, we obtain in Sec. 3 for the analogous
curvature properties of the  $\overline{\W}_6$-manifolds same or
similar algebraic expressions.

\section{Preliminaries}

Let $(M,P,g)$ be a \emph{Riemannian almost product manifold}, \ie
a differentiable manifold $M$ with a tensor field $P$ of type
$(1,1)$ and a Riemannian metric $g$ such that $P^2x=x$,
$g(Px,Py)=g(x,y)$ for any $x$, $y$ of the algebra $\X(M)$ of the
smooth vector fields on $M$. Further $x,y,z,u,w$ will stand for
arbitrary elements of $\X(M)$ or vectors in the tangent space
$T_cM$ at $c\in M$.

In this work we consider manifolds $(M, P, g)$ with $\tr{P}=0$. In
this case $M$ is an even-dimensional manifold. We assume that
$\dim{M}=2n$.

In \cite{Na} A.M.~Naveira gives a classification of Riemannian
almost pro\-duct manifolds with respect to the tensor $F$ of type
(0,3), defined by
\begin{equation}\label{1}
F(x,y,z)=g\left(\left(\nabla_x P\right)y,z\right),
\end{equation}
where $\n$ is
the Levi-Civita connection of $g$. The tensor $F$ has the
properties:
\begin{equation*}
    F(x,y,z)=F(x,z,y)=-F(x,Py,Pz),\quad
    F(x,y,Pz)=-F(x,Py,z).
\end{equation*}

Using the Naveira classification, in \cite{StaGri} M.~Staikova and
K.~Gribachev give a classification of Riemannian almost product
manifolds $(M,P,g)$ with $\tr P=0$. The basic classes of this
classification are $\W_1$, $\W_2$ and $\W_3$. Their intersection
is the class $\W_0$ of the \emph{Riemannian $P$-manifolds}
(\cite{Sta}), determined by the condition $F=0$. This class is an
analogue of the class of K\"ahler manifolds in the geometry of
almost Hermitian manifolds.

The class $\W_1$ from the Staikova-Gribachev classification
contains the manifolds which are locally conformal equivalent to
Riemannian $P$-manifolds. This class plays a similar role of the
role of the class of the conformal K\"ahler manifolds in almost
Hermitian geometry. We will say that a manifold from the class
$\W_1$ is a \emph{$\W_1$-manifold}.

The characteristic condition for the class $\W_1$ is the following
\begin{equation*}
\begin{array}{l}
\W_1: F(x,y,z)=\frac{1}{2n}\big\{ g(x,y)\ta (z)-g(x,Py)\ta (Pz)
 \big.\\[4pt]
 \phantom{\W_1: F(x,y,z)=\frac{1}{2n}} +g(x,z)\ta (y)-g(x,Pz)\ta (Py)\big\},
\end{array}
\end{equation*}
where the associated 1-form $\ta$ is determined by $
\ta(x)=g^{ij}F(e_i,e_j,x). $ Here $g^{ij}$ will stand for the
components of the inverse matrix of $g$ with respect to a basis
$\{e_i\}$ of $T_cM$ at $c\in M$. The 1-form $\ta$ is
\emph{closed}, \ie $\D\ta=0$, if and only if
$\left(\n_x\ta\right)y=\left(\n_y\ta\right)x$. Moreover, $\ta\circ
P$ is a closed 1-form if and only if
$\left(\n_x\ta\right)Py=\left(\n_y\ta\right)Px$.

In \cite{StaGri} it is proved that
$\W_1=\overline\W_3\oplus\overline\W_6$, where $\overline\W_3$ and
$\overline\W_6$ are the classes from the Naveira classification
determined by the following conditions:
\[
\begin{array}{rl}
\overline\W_3:& \quad F(A,B,\xi)=\frac{1}{n}g(A,B)\ta^v(\xi),\quad
F(\xi,\eta,A)=0,
\\[4pt]
\overline\W_6:& \quad
F(\xi,\eta,A)=\frac{1}{n}g(\xi,\eta)\ta^h(A),\quad F(A,B,\xi)=0,
\end{array}
\]
where $A,B,\xi,\eta\in\X(M)$, $PA=A$, $PB=B$, $P\xi=-\xi$,
$P\eta=-\eta$, $\ta^v(x)=\frac{1}{2}\left(\ta(x)-\ta(Px)\right)$,
$\ta^h(x)=\frac{1}{2}\left(\ta(x)+\ta(Px)\right)$. In the case
when $\tr P=0$, the above conditions for $\overline\W_3$ and
$\overline\W_6$ can be written for any $x,y,z$ in the following
form:
\begin{equation}\label{2}
\begin{array}{rl}
    \overline\W_3: \quad
    &F(x,y,z)=\frac{1}{2n}\bigl\{\left[g(x,y)+g(x,Py)\right]\ta(z)\\[4pt]
    &+\left[g(x,z)+g(x,Pz)\right]\ta(y)\bigr\},\quad
    \ta(Px)=-\ta(x),
\end{array}
\end{equation}
\begin{equation}\label{3}
\begin{array}{rl}
    \overline\W_6: \quad
    &F(x,y,z)=\frac{1}{2n}\bigl\{\left[g(x,y)-g(x,Py)\right]\ta(z)\\[4pt]
    &+\left[g(x,z)-g(x,Pz)\right]\ta(y)\bigr\},\quad
    \ta(Px)=\ta(x).
\end{array}
\end{equation}

We will say that a manifold from the class $\overline{\W}_3$
(resp., $\overline{\W}_6$) is a \emph{$\overline{\W}_3$-manifold}
(resp., \emph{$\overline{\W}_6$-manifold}).

In \cite{StaGri}, a tensor $L$ of type (0,4) with pro\-per\-ties%
\begin{equation}\label{4}
L(x,y,z,w)=-L(y,x,z,w)=-L(x,y,w,z),
\end{equation}
\begin{equation}\label{5}
L(x,y,z,w)+L(y,z,x,w)+L(z,x,y,w)=0
\end{equation}
is called a \emph{curvature-like tensor}. Such a tensor on a
Riemannian almost product manifold $(M,P,g)$ with the property
\begin{equation}\label{6}
L(x,y,Pz,Pw)=L(x,y,z,w)
\end{equation}
is called a \emph{Riemannian $P$-tensor} in \cite{Mek}. This
notion is an analogue of the notion of a K\"ahler tensor in
Hermitian geometry.

Let $S$ be a (0,2)-tensor on a Riemannian almost product manifold.
In \cite{StaGri} it is proved that
\[
\begin{split}
\psi_1(S)(x,y,z,w)
&=g(y,z)S(x,w)-g(x,z)S(y,w)\\[4pt]
&+S(y,z)g(x,w)-S(x,z)g(y,w)
\end{split}
\]
is a curvature-like  tensor if and only if $S(x,y)=S(y,x)$, and
the tensor $$\psi_2(S)(x,y,z,w)=\psi_1(S)(x,y,Pz,Pw)$$ is
curvature-like if and only if $S(x,Py)=S(y,Px)$. Obviously
$$\psi_2(S)(x,y,Pz,Pw)=\psi_1(S)(x,y,z,w).$$ The tensors
\[
\pi_1=\frac{1}{2}\psi_1(g),\qquad
\pi_2=\frac{1}{2}\psi_2(g),\qquad
\pi_3=\psi_1(\widetilde{g})=\psi_2(\widetilde{g}),
\]
where $\tg(x,y)=g(x,Py)$, are curvature-like, and the tensors
$\pi_1+\pi_2$, $\pi_3$ are Riemannian $P$-tensors.

Let us recall the following statement.

\begin{thm}[\cite{DobrMek}]\label{thm-2.1}
A curvature-like tensor $L$ on 4-dimensional Riemannian almost
product manifold is a Riemannian $P$-tensor if and only if $L$ has
the following form:
\begin{equation*}
    L=\frac{1}{8}\left\{\tau(L)(\pi_1+\pi_2)+\tau^*(L)\pi_3\right\}.
\end{equation*}
\end{thm}

 The curvature tensor $R$ of $\n$ is determined by
$R(x,y)z=\nabla_x \nabla_y z - \nabla_y \nabla_x z -
    \nabla_{[x,y]}z$ and the corresponding tensor of type (0,4) is defined as
follows $R(x,y,z,w)=g(R(x,y)z,w)$. We denote the Ricci tensor and
the scalar curvature of $R$ by $\rho$ and $\tau$, respectively,
\ie $\rho(y,z)=g^{ij}R(e_i,y,z,e_j)$ and
$\tau=g^{ij}\rho(e_i,e_j)$. The associated Ricci tensor $\rho^*$
and the associated scalar curvature $\tau^*$ of $R$ are determined
by $\rho^*(y,z)=g^{ij}R(e_i,y,z,Pe_j)$ and
$\tau^*=g^{ij}\rho^*(e_i,e_j)$. In a similar way there are
determined the Ricci tensor $\rho(L)$ and the scalar curvature
$\tau(L)$ for any curvature-like tensor $L$ as well as the
associated quantities $\rho^*(L)$ and $\tau^*(L)$.

It is known the Ricci identity for an almost product structure in
the following form
\[
R(x,y)Pz-PR(x,y)z=\left(\n_x\n_y P\right)z-\left(\n_y\n_x
P\right)z.
\]

Taking into account \eqref{1} and $\n g=0$, the latter identity
implies immediately
\begin{equation}\label{7}
\begin{split}
    R(x,y,Pz,w)-R(x,y,z,Pw)&=\left(\n_xF\right)(y,z,w)\\[4pt]
    &-\left(\n_yF\right)(x,z,w).
\end{split}
\end{equation}

\section{Curvature properties of $\overline\W_3$-manifolds}

Let $(M,P,g)$ is a $\overline\W_3$-manifold, i.e. \eqref{2} is
valid.

The first equality of \eqref{2} can be rewritten in the form
\begin{equation}\label{8}
    \left(\n_yP\right)z=\frac{1}{2n}\left\{\left[g(y,z)+g(y,Pz)\right]\Om+(y+Py)\ta(z)\right\},
\end{equation}
where the vector $\Om$ is defined by $g(\Om,x)=\ta(x)$. By
differentiation of the condition $\ta(Pz)=-\ta(z)$ and taking into
account \eqref{8} and $\ta(y+Py)=0$, we obtain
\begin{equation}\label{9}
    \left(\n_y\ta\right)Pz=-\left(\n_y\ta\right)z-\frac{\ta(\Om)}{2n}\left[g(y,z)+g(y,Pz)\right].
\end{equation}

By differentiation of the equality $\tg(y,z)=g(y,Pz)$ we get
\begin{equation}\label{10}
    \left(\n_x\tg\right)(y,z)=F(x,y,z).
\end{equation}

We differentiate the first equality of
 \eqref{2} and take into account
 \eqref{7},  \eqref{10} and $\n g=0$. Thus we obtain
 \[
\begin{split}
    &R(x,y,Pz,w)-R(x,y,z,Pw)=\\[4pt]
    &=\left[g(y,z)+g(y,Pz)\right]\left(\n_x\ta\right)w
    -\left[g(x,z)+g(x,Pz)\right]\left(\n_y\ta\right)w\\[4pt]
    &+\left[g(y,w)+g(y,Pw)\right]\left(\n_x\ta\right)z%
    -\left[g(x,w)+g(x,Pw)\right]\left(\n_y\ta\right)z\\[4pt]
    &+\left[F(x,y,z)-F(y,x,z)\right]\ta(x)+\left[F(x,y,w)-F(y,x,w)\right]\ta(z).
\end{split}
 \]

Thus, bearing in mind the first equality of \eqref{2}, we have
\begin{equation}\label{11}
\begin{split}
    &R(x,y,Pz,w)-R(x,y,z,Pw)=\\[4pt]
    &=\frac{1}{2n}\bigl\{\left[g(y,z)+g(y,Pz)\right]A(x,w)\\[4pt]
    &\phantom{=\frac{1}{2n}\bigl\{}-\left[g(x,z)+g(x,Pz)\right]A(y,w)\\[4pt]
    &\phantom{=\frac{1}{2n}\bigl\{}+\left[g(y,w)+g(y,Pw)\right]A(x,z)\\[4pt]%
    &\phantom{=\frac{1}{2n}\bigl\{}-\left[g(x,w)+g(x,Pw)\right]A(y,z)\bigr\},
\end{split}
\end{equation}
where
\begin{equation}\label{12}
    A(y,z)=\left(\n_y\ta\right)z-\frac{1}{2n}\ta(y)\ta(z).
\end{equation}

\begin{lem}\label{lem-3.1}
The following statements are equivalent for a
$\overline\W_3$-manifold $(M,P,g)$:
\begin{enumerate}
    \item[(i)] $\ta$ ia a closed 1-form;
    \item[(ii)] $A(y,z)=A(z,y)$;
    \item [(iii)]$A(y,Pz)=A(z,Py)$, $A(Py,Pz)=A(y,z)$.
\end{enumerate}
\end{lem}
\begin{proof}
From \eqref{2} and \eqref{9} we obtain
\begin{equation}\label{13}
    A(y,Pz)=-A(y,z)-\frac{\ta(\Om)}{2n}\left[g(y,z)+g(y,Pz)\right].
\end{equation}
Equalities \eqref{12} and \eqref{13} imply the truthfulness of the
lemma.
\end{proof}

We substitute $Pw$ for $w$ in \eqref{11}. By the obtained
equality, \eqref{12}, \eqref{13} and \lemref{lem-3.1}, we get the
following
\begin{thm}\label{thm-3.2}
The following equality is valid for a $\overline\W_3$-manifold
$(M,P,g)$:
\begin{equation}\label{14}
\begin{split}
    &R(x,y,Pz,Pw)-R(x,y,z,w)=\\[4pt]
    &=-\frac{1}{2n}\left\{(\psi_1-\psi_2)(A)(x,y,z,w)+\frac{\ta(\Om)}{2n}(\pi_1-\pi_2)(x,y,z,w)\right\}.
\end{split}
\end{equation}
\end{thm}

According to \lemref{lem-3.1}, the 1-form $\ta$ is closed if and
if $\psi_1(A)$ and $\psi_2(A)$ are curvature-like tensors. In this
case, since $\pi_1$ and $\pi_2$ are also curvature-like tensors,
equality \eqref{14} implies the following

\begin{thm}\label{thm-3.3}
The 1-form $\ta$ is closed on a $\overline\W_3$-manifold $(M,P,g)$
if and only if the following equality is valid:
\begin{equation}\label{15}
    R(x,y,Pz,Pw)+R(y,z,Px,Pw)+R(z,x,Py,Pw)=0.
\end{equation}
\end{thm}

\begin{cor}\label{cor-3.4}
If the 1-form $\ta$ is closed on a $\overline\W_3$-manifold
$(M,P,g)$ then the following equalities are valid:
\begin{equation}\label{15'}
    R(Px,Py,Pz,Pw)=R(x,y,z,w),\qquad \rho(Py,Pz)=\rho(y,z).
\end{equation}
\end{cor}

Further we consider a $\overline\W_3$-manifold $(M,P,g)$ with a
closed 1-form $\ta$.

We define a tensor $K$ of type (0,4) by
\begin{equation}\label{16}
    K(x,y,z,w)=\frac{1}{2}\left\{R(x,y,z,w)+R(x,y,Pz,Pw)\right\}.
\end{equation}

\begin{prop}\label{prop-3.5}
If $(M,P,g)$ is a $\overline\W_3$-manifold with closed 1-form
$\ta$ then $K$ is a Riemannian $P$-tensor and it has the form
\begin{equation}\label{17}
    K=R-\frac{1}{4n}\left\{(\psi_1-\psi_2)(A)+\frac{\ta(\Om)}{2n}(\pi_1-\pi_2)\right\}.
\end{equation}
\end{prop}

\begin{proof}
Obviously, properties \eqref{4} and \eqref{6} are satisfied for
$K$. Moreover, property \eqref{5} is valid because of \eqref{15}.
Therefore, $K$ is a Riemannian $P$-tensor. Equality \eqref{17}
follows from \eqref{14} and \eqref{16}.
\end{proof}

\begin{cor}\label{cor-3.6}
If $(M,P,g)$ is a $\overline\W_3$-manifold with closed 1-form
$\ta$ then for the Ricci tensor and the scalar curvatures of $K$
the following formulae are valid:
\[
\rho(K)(y,z)=\rho(y,z)-\frac{1}{4n}\left\{\left[\tr
A+\ta(\Om)\right]\left[g(y,z)+g(y,Pz)\right]+2nA(y,z)\right\},
\]
\[
\tau(K)=\tau-\Div\Om-\frac{n-1}{2n}\ta(\Om),\qquad
\tau^*(K)=\tau^*,
\]
where $\Div\Om$ is the divergence of the vector $\Om$.
\end{cor}

By virtue of \thmref{thm-2.1} and \corref{cor-3.6} we obtain the
following
\begin{thm}\label{thm-3.7}
If $(M,P,g)$ is a 4-dimensional $\overline\W_3$-manifold with
closed 1-form $\ta$ then the tensor $K$ has the form:
\begin{equation}\label{17'}
    K=\frac{1}{8}\left\{\left[\tau+\Div\Om-\frac{\ta(\Om)}{4}\right](\pi_1+\pi_2)+\tau^*\pi_3\right\}.
\end{equation}
\end{thm}

Now we will find a necessary and sufficient condition the tensor
$R$ to be a Riemannian $P$-tensor on a $\overline\W_3$-manifold
with closed 1-form $\ta$.

If $R$ is a Riemannian $P$-tensor on a $\overline\W_3$-manifold
with closed 1-form $\ta$ then the left-hand side of \eqref{14} is
zero, i.e. we have
\[
    (\psi_1-\psi_2)(A)+\frac{\ta(\Om)}{2n}(\pi_1-\pi_2)=0.
\]
Then, according to \eqref{13}, it follows
\begin{equation}\label{18}
    \left[\tr
A+\ta(\Om)\right]\left[g(y,z)+g(y,Pz)\right]+2nA(y,z)=0.
\end{equation}
From \eqref{18} we obtain $\tr A=-\frac{\ta(\Om)}{2}$ and then
again from \eqref{18} it follows
\begin{equation}\label{19}
    A(y,z)=-\frac{\ta(\Om)}{4n}\left[g(y,z)+g(y,Pz)\right].
\end{equation}

Vice versa, let condition \eqref{19} be satisfied for a
$\overline\W_3$-manifold with closed 1-form $\ta$. Then
\[
    \psi_1(A)=-\frac{\ta(\Om)}{4n}(2\pi_1+\pi_3),\qquad
    \psi_2(A)=-\frac{\ta(\Om)}{4n}(2\pi_2+\pi_3)
\]
and thus \eqref{14} implies condition \eqref{6} for $R$, i.e. $R$
is a Riemannian $P$-tensor.

Therefore, it is valid the following
\begin{thm}\label{thm-3.8}
Let $(M,P,g)$ be a $\overline\W_3$-manifold with closed 1-form
$\ta$. Then $R$ is a Riemannian $P$-tensor if and only if
condition \eqref{19} is valid.
\end{thm}

From \propref{prop-3.5} and \thmref{thm-3.8} we obtain immediately
the following
\begin{cor}\label{cor-3.9}
Let $(M,P,g)$ be a $\overline\W_3$-manifold with closed 1-form
$\ta$. Then $R$ is a Riemannian $P$-tensor if and only if $R=K$.
In this case, if $\dim M=4$ the tensor $R$ has the form
\begin{equation}\label{20}
    R=\frac{1}{8}\left\{\tau(\pi_1+\pi_2)+\tau^*\pi_3 \right\}.
\end{equation}
\end{cor}

\section{Curvature properties of of $\overline\W_6$-manifolds}

Let $(M,P,g)$ is a $\overline\W_6$-manifold, i.e. \eqref{3} is
valid. Following the approach of the previous section, now we
consider problems for a $\overline\W_6$-manifold analogous to the
considered ones for a $\overline\W_3$-manifold.

The first equality of \eqref{3} can be rewritten in the form
\begin{equation*}
    \left(\n_yP\right)z=\frac{1}{2n}\left\{\left[g(y,z)-g(y,Pz)\right]\Om+(y-Py)\ta(z)\right\},
\end{equation*}
from where we obtain
\begin{equation*}
    \left(\n_y\ta\right)Pz=\left(\n_y\ta\right)z-\frac{\ta(\Om)}{2n}\left[g(y,z)-g(y,Pz)\right].
\end{equation*}

Let $A'$ be the tensor determined by
\begin{equation*}
    A'(y,z)=\left(\n_y\ta\right)z+\frac{1}{2n}\ta(y)\ta(z),
\end{equation*}
and $K$ is the tensor determined by \eqref{16}.

We establish the truthfulness of the following statements.

\begin{lem}\label{lem-4.1}
The following statements are equivalent for a
$\overline\W_6$-manifold $(M,P,g)$:
\begin{enumerate}
    \item[(i)] $\ta$ ia a closed 1-form;
    \item[(ii)] $A'(y,z)=A'(z,y)$;
    \item [(iii)]$A'(y,Pz)=A'(z,Py)$, $A'(Py,Pz)=A'(y,z)$.
\end{enumerate}
\end{lem}

\begin{thm}\label{thm-4.2}
The following equality is valid for a $\overline\W_6$-manifold
$(M,P,g)$:
\begin{equation*}
\begin{split}
    &R(x,y,Pz,Pw)-R(x,y,z,w)=\\[4pt]
    &=\frac{1}{2n}\left\{(\psi_1-\psi_2)(A')(x,y,z,w)-\frac{\ta(\Om)}{2n}(\pi_1-\pi_2)(x,y,z,w)\right\}.
\end{split}
\end{equation*}
\end{thm}

\begin{thm}\label{thm-4.3}
The 1-form $\ta$ is closed on a $\overline\W_6$-manifold $(M,P,g)$
if and only if the equality \eqref{15} is valid.
\end{thm}

\begin{cor}\label{cor-4.4}
If the 1-form $\ta$ is closed on a $\overline\W_6$-manifold
$(M,P,g)$ then the equalities \eqref{15'} are valid.
\end{cor}

\begin{prop}\label{prop-4.5}
If $(M,P,g)$ is a $\overline\W_6$-manifold with closed 1-form
$\ta$ then $K$ is a Riemannian $P$-tensor and it has the form
\begin{equation*}
    K=R+\frac{1}{4n}\left\{(\psi_1-\psi_2)(A')-\frac{\ta(\Om)}{2n}(\pi_1-\pi_2)\right\}.
\end{equation*}
\end{prop}

\begin{cor}\label{cor-4.6}
If $(M,P,g)$ is a $\overline\W_6$-manifold with closed 1-form
$\ta$ then for the Ricci tensor and the scalar curvatures of $K$
the following formulae are valid:
\[
\begin{split}
\rho(K)(y,z)&=\rho(y,z)\\[4pt]
&+\frac{1}{4n}\left\{\left[\tr
A'-\ta(\Om)\right]\left[g(y,z)-g(y,Pz)\right]+2nA'(y,z)\right\},
\end{split}
\]
\[
\tau(K)=\tau+\Div\Om-\frac{n-1}{2n}\ta(\Om),\qquad
\tau^*(K)=\tau^*.
\]
\end{cor}

\begin{thm}\label{thm-4.7}
If $(M,P,g)$ is a 4-dimensional $\overline\W_6$-manifold with
closed 1-form $\ta$ then the tensor $K$ has the form \eqref{17'}.
\end{thm}

\begin{thm}\label{thm-4.8}
Let $(M,P,g)$ be a $\overline\W_6$-manifold with closed 1-form
$\ta$. Then $R$ is a Riemannian $P$-tensor if and only if the
following  condition is valid:
\begin{equation*}
    A'(y,z)=\frac{\ta(\Om)}{4n}\left[g(y,z)-g(y,Pz)\right].
\end{equation*}
\end{thm}

\begin{cor}\label{cor-4.9}
Let $(M,P,g)$ be a $\overline\W_6$-manifold with closed 1-form
$\ta$. Then $R$ is a Riemannian $P$-tensor if and only if $R=K$.
In this case, if $\dim M=4$ the tensor $R$ has the form
\eqref{20}.
\end{cor}

\bigskip

\small{ \noindent
\textsl{D. Gribacheva\\
Department of Algebra and Geometry\\
Faculty of Mathematics and Informatics\\
University of Plovdiv\\
236 Bulgaria Blvd\\
4003 Plovdiv, Bulgaria}
\\
\texttt{dobrinka@uni-plovdiv.bg} }

\end{document}